\documentclass{article}
\usepackage[latin1]{inputenc}
\usepackage[small]{titlesec}
\usepackage{amsmath}
\usepackage{amsthm}
\usepackage{amssymb}
\usepackage{graphicx}
\usepackage{standalone}
\usepackage{color}
\theoremstyle{plain}
\newtheorem{thm}{Theorem}[section]
\newtheorem{lem}[thm]{Lemma}

\newtheorem{cor}[thm]{Corollary}

\newtheorem{rem}[thm]{Remark}
\newtheorem{defn}[thm]{Definition}

\usepackage[small]{titlesec}
\usepackage{blindtext}

	\title{All minimal $[9,4]_{2}$-codes are hyperbolic quadrics}
\author{Valentino Smaldore \footnote{Valentino Smaldore:
valentino.smaldore@unibas.it 
Dipartimento di Matematica, Informatica ed Economia,
Universit\`{a} degli Studi della Basilicata, Viale dell'Ateneo Lucano 10, 85100 Potenza, Italy.}}

\date{}

\begin{document}
\maketitle
\begin{abstract}
Minimal codes are being intensively studied in last years. $[n,k]_{q}$-minimal linear codes are in bijection with strong blocking sets of size $n$ in $PG(k-1,q)$ and a lower bound for the size of strong blocking sets is given by $(k-1)(q+1)\leq n$. In this note we show that all strong blocking sets of length 9 in $PG(3,2)$ are the hyperbolic quadrics $Q^{+}(3,2)$.

\end{abstract}

\section{Introduction}
\subsection{First definitions}
In this section we fix the notation that we will use in the paper and some preliminary results on minimal codes.
\begin{defn}
  \begin{itemize}
   \item An $[n,k]_{q}$-\textit{linear code} $\mathcal{C}$ is a subspace of $\mathbb{F}_{q}^{n}$ of dimension $k$.
   \item All the elements of $\mathcal{C}$ are said \textit{codewords}.
   \item The \textit{(Hamming) support} of a vector $(v_{1},v_{2},\ldots,v_{n})=v\in \mathbb{F}_{q}^{n}$ is\\ $\sigma(v)=\{i|v_{i}\neq0\}\subseteq\{1,2,\ldots,n\}$.
   \item The \textit{(Hamming) weight} of $v$ is $w(v)=|\sigma(v)|$.
   \item A \textit{Generator matrix} $G$ for $\mathcal{C}$ is an $k\times n$ matrix over $\mathbb{F}_{q}$ such that\\ $\mathcal{C}=rowspace(G)$ .
  \end{itemize}
\end{defn}
In \cite{Massey}, J. L. Massey used \textit{minimal codewords} to determine the access structure in its codebased secrete sharing scheme.
\begin{defn}
 Let $\mathcal{C}$ be an $[n,k]_{q}$ code. A nonzero codeword $c\in\mathcal{C}$ is \textit{minimal} if for every codeword $c'\in\mathcal{C}$ such that $\sigma(c')\subseteq\sigma(c)$, $c'=\lambda c$, for some $\lambda$ in $\mathbb{F}^*_{q}$. We say that $\mathcal{C}$ is \textit{minimal} if all its codewords are minimal.
\end{defn}

As it was noticed in \cite{A2}, minimal codes are in bijection with strong blocking sets in projective spaces. A strong blocking set in $PG(k-1,q)$ is a set of points $\mathcal{M}$ such that, for each hyperplane $\sigma$, $\langle\sigma\cap\mathcal{M}\rangle=\sigma$. Strong blocking sets were introduced by A. A. Davydov, M. Giulietti, S. Marcugini and F. Pambianco in \cite{DGMP}. Later, the same concept was also studied by Sz. Fancsali and P. Sziklai in \cite{HP}, under the name of \textit{generator sets}, by M. Bonini and M. Borello in \cite{BB}, under the name of \textit{cutting  blocking sets}, and by T. H\'{e}ger, B. Patk\'{o}s and M. Tak\'{a}ts in \cite{Heger}, under the name of \textit{hyperplane generating set}.

\begin{thm}\cite[Theorem 3.4]{A2}
 Let $\mathcal{C}$ be a non-degenerate $[n,k]_{q}$-code with generator matrix $G=(G_{1},\ldots ,G_{n})$. Let $S=\{G_{1},\ldots ,G_{n}\}$ be the corresponding point set of $PG(k-1, q)$. Then, $\mathcal{C}$ is a minimal code if and only if $S$ is a strong blocking set.
\end{thm}
From the previous theorem it follows the correspondence between \\$[n,k]_{q}$-minimal code and strong blocking sets of $PG(k-1,q)$ of size $n$. We refer the reader to \cite{A2} for more details on minimal codes, while we refer to \cite{Heger} for a more geometrical approach.

\subsection{Minimal length of a minimal code}
A general problem is to determine the minimal length of a $[n,k]_{q}$-minimal code can have, fixing $k$ and $q$. Equivalently, the problem is to determine the minimum size $n$ of a strong blocking set in $PG(k-1,q)$. In $\cite{Beut}$, A. Beutelspacher found a lower bound for $n$, under the hypothesis $k-1\leq q$. Later, G. N. Alfarano, M. Borello, A. Neri and A. Ravagnani extended the result for all values of $k$ and $q$.
\begin{thm}\cite[Theorem 2.14]{A1}
\label{main}
 Let $\mathcal{C}$ be a $[n,k]_{q}$-minimal code. Then $n\geq(k-1)(q+1)$.
\end{thm}
\begin{cor}
 \label{coroll}
 Let $q=2$. If $n$ is the size of a strong blocking set in $PG(k-1,2)$, then $n\geq3(k-1)$.
\end{cor}

\section{Minimal strong blocking sets in $PG(3,2)$}
 Consider now the projective space $PG(3,2)$. From Corollary \ref{coroll} we know that the minimum possible size for a strong blocking set is 9. In \cite[Example 5.10]{A2}, is pointed out that a computer search showed that there is only one minimal code of length 9 up to equivalence. Here we prove in terms of minimal strong blocking sets of $PG(3,2)$. We call \textit{minimal strong blocking set} a strong blocking set $S$ meeting te
 the lower bound given by Theorem \ref{main}.

 It is easy to see that hyperbolic quadrics $Q^{+}(3,2)$ are minimal strong blocking sets, i.e. all planar intersections generate the whole plane. The proof is contained also in \cite{A2}, and belongs to a more general result in \cite{Timpa}. In fact, it is w
 known that plane section of a quadric is a non-degenerate conic, or a degenerate conic consisting on two concurrent lines, and in both cases the whole plane is generated by the intersection.

We now focus on a geometric characterization of minimal strong blocking sets of $PG(3,2)$
\begin{lem}
 \label{lemma1}
 Consider a minimal strong blocking set $S$ of $PG(3,2)$.
 \begin{itemize}
  \item Any projective plane $\pi$ in $PG(3,2)$ contains at most 5 points of $S$.
  \item If $\pi$ contains a line of $S$, then it contains 5 points of $S$, forming two lines meeting in a point.
 \end{itemize}
\end{lem}
\begin{proof}
 \begin{itemize}
  \item Fix a plane $\pi$. Since $|S|=9$, we want to prove that $|S\setminus\pi|>3$. Assume by contradiction that $|S\setminus\pi|\leq3$, then consider the plane $\sigma$ through the points in $S\setminus\pi$. The intersection of $\sigma$ and $\pi$ is a line $\ell$, and through a $\ell$ there pass exactly three planes: $\sigma$, $\pi$ and $\delta$. But now $\delta\cap S\subseteq\ell$, which does not generate the whole plane $\delta$.
  \item Consider a line $\ell\subset S$. Through $\ell$ there pass the 3 planes $\sigma$, $\pi$ and $\delta$. Since by the previous a plane may not contain more than 5 points of $S$, and since $|S|=9$, we have $|S\cap\pi|=|S\cap\sigma|=|S\cap\delta|=5$.
 \end{itemize}
\end{proof}
\begin{lem}
 \label{lemma2}
 Let $S$ be a minimal strong blocking set of $PG(3,2)$. Through each point $P\in S$ there pass exactly 2 lines entirely contained in $S$.
\end{lem}
\begin{proof}
 Firstly we prove that through $P$ may not pass more than 2 lines of $S$. Otherwise, say $\ell_1,\ell_2,\ell_3$ are 3 lines through $P$ contained in $S$. Now, $|\ell_1\cup\ell_2\cup\ell_3|=7$, and the other 2 points of the minimal strong blocking set $S$ are on two more lines $\ell_4$ and $\ell_5$ through $P$. Since $P$ lies on 7 lines, take $\ell_6$ and $\ell_7$ meeting $S$ only in $P$. Now the plane $\langle\ell_6,\ell_7\rangle$ is such that $S\cap\langle\ell_6,\ell_7\rangle\in\{S\cap\ell_1,S\cap\ell_2,S\cap\ell_3,S\cap\ell_4,S\cap\ell_5\}$, and in all these cases the intersection does not generate $\langle\ell_6,\ell_7\rangle$.
 The second step is to prove that through each $P\in S$ there pass at least two lines of $S$. Since $P$ lies on $7$ lines and $|S\setminus\{P\}|=8$, by counting we see that at least one of the lines through $P$ is entirely contained in $S$ (recall that a projective line in $PG(3,2)$ has 3 points). Now fix a line $\ell\subset S$, and the 3 planes $\pi_1$, $\pi_2$ and $\pi_3$ through $\ell$. Since by Lemma \ref{lemma1} $|S\cap\pi_1|=|S\cap\pi_2|=|S\cap\pi_3|=5$, each of the planes through $\ell$ contains an other line in $S$. But we just proved that a point may not lie on 3 lines in $S$, so taking the 3 points $P_1,P_2,P_3\in\ell$, we see that, up to a possible permutation on points, through $P_i$ there pass exactly an other line in $S\cap\pi_i$, $i\in\{1,2,3\}$.
 \end{proof}

 We are now ready to prove the main theorem.
\begin{thm}
 The minimal strong blocking sets of $PG(3,2)$ are exactly the hyperbolic quadrics $Q^{+}(3,2)$.
\end{thm}
\begin{proof}
 The thesis arises from the characterization given in Lemma \ref{lemma1} and Lemma \ref{lemma2}, together with the known classification of quadrics in \cite{Hirschfeld3}\footnote{We thank the reviewer for the shortcut in the proof of the main theorem}.
\end{proof}

\begin{rem}
 An other proof of the theorem, which does not involve the classification of polar spaces, arises by checking all projectively non-equivalent configurations of $)$ points in $PG(3,2)$. Since $|PG(3,2)|=15$, the set $X$ of all possible configurations of 9 points in $PG(3,2)$ has size $|X|=\binom{15}{9}=5005$. In $PG(3,2)$ there are exactly 5 projectively non-equivalent configurations of 9 points, i.e. $PGL(4,2)$ has exactly 5 orbits on $X$:
\begin{enumerate}
  \item 280 hyperbolic quadrics $Q^{+}(3,2)$;
  \item 105 configurations that consist on a point $P$ and all the other points out of a fixed plane $\pi\ni P$;
  \item 420 configurations that consist on a plane $\pi$ and two points $P,Q\notin\pi$;
  \item 1680 configurations that consist on a punctured plane $\pi\setminus\{P\}$ and other 3 points $Q,R,S\notin\pi$ such that $\langle Q,R,S\rangle\cap\pi=\ell\not\ni P$;
  \item 2520 configurations that consist on a punctured plane $\pi\setminus\{P\}$ and other 3 points $Q,R,S\notin\pi$ such that $\langle Q,R,S\rangle\cap\pi=\ell\ni P$.
 \end{enumerate}
By the \textit{Orbit-Stabilizer Theorem} we have $\frac{|PGL(4,2)|}{|PGO^{+}(4,2)|}=\frac{20160}{72}=280$ hyperbolic quadrics. Now we consider the second configuration. The intersection with the plane $\pi$ consists of the singleton $\{P\}$, while the intersections with the other planes have size 4 or 5, so we have no equivalence with the hyperbolic quadrics. The number of point in $|PG(3,2)|$ is 15 and through each point there pass 7 planes, so the second set has size $7\cdot15=105$. The third configuration considers the 15 planes and the all possible pairs of points in the remaining $15-7=|PG(3,2)\setminus\pi|$, so we get $15\cdot\binom{8}{2}=420$. Now we consider all the possibilities made of 6 points on a plane $\pi$ and three points $Q,R,S\notin\pi$. Since the intersection with $\pi$ is 6, the configuration is non-equivalent to all the other considered.
We have 15 planes, and each time we take out $P$, one of the 7 points of the plane. Moreover, we should consider all the possible triple of the 8 points out of $\pi$, so the total number is $15\cdot7\cdot\binom{8}{3}=5880$. here we count twice the 1680 configurations as in 4, and we have 2520 configurations as in 5. But now we have to take care on two different cases: if $P\in\ell=\langle Q,R,S\rangle\cap\pi$ the intersection with $\langle Q,R,S\rangle$ has size 5, while if $P\notin\ell=\langle Q,R,S\rangle\cap\pi$, the intersection with $\langle Q,R,S\rangle$ consists on $\ell\cup\{Q,R,S\}$ and has size 6, so we have counted twice this configuration in the 5880 possibilities. An other way to consider the latter, is to fix the line $\ell$ and fix two of the three planes through $\ell$, taking out one point from each plane. Since we have 35 lines in $PG(3,2)$ and we consider all the $\binom{3}{2}=3$ pairs of planes $\pi,\pi'$ through $\ell$, excluding one of the 4 points of $\pi\setminus\ell$ and one of the 4 points of $\pi'\setminus\ell$, the total number is $35\cdot3\cdot4\cdot4=1680$. The last configuration, when $\langle Q,R,S\rangle\cap\pi=\ell\ni P$, contains $5880-(2\cdot1680)=2520$ sets. We are actually considering all the possibilities since $5005=280+105+420+1680+2520$.
\end{rem}

 \section{Conclusion}
  In this paper we provided examples of strong blocking set of the minimal possible size allowed by Theorem \ref{main} in $PG(3,2)$. From Corollary \ref{coroll} we know that the size of a minimal strong blocking set in a projective $(k-1)$-dimensional space over $\mathbb{F}_{2}$ is $3(k-1)$. In the \textit{Fano plane} we find the trivial 7 configurations of 6 points given by all the point of the plane except one. The following step should rely on $PG(4,2)$. In this case $n\geq12$, while, for example, a parabolic quadric $Q(4,2)$ is a strong blocking set of size 15. Computer-aided search does not show any configuration of 12 points such that the intersection with all hyperplanes generate the whole 3-space, while it is possible to show minimal strong blocking set of size 15 in $PG(5,2)$.

\end{document}